\documentclass[10pt]{amsart}

\usepackage{amsfonts, amsmath, amsthm, amssymb, bbm}

\newtheorem{theorem}{Theorem}

\newtheorem{corollary}[theorem]{Corollary}
\newtheorem{proposition}[theorem]{Proposition}

\theoremstyle{definition}

\newtheorem{fact}{Fact}[section]
\newtheorem{claim}{Claim}[section]

\newenvironment{proofofclaim}{{\itshape Proof of Claim.}}{\hfill\copyright}

\theoremstyle{remark}

\def\cf{{\mathrm{cf} \, }}

\def\rest{{\upharpoonright}}

\begin{document}

\title{Selectivity properties of spaces}


\author{William Chen-Mertens}
\address{York University}
\email{chenwb@gmail.com}

\author{Paul J. Szeptycki}
\address{York University}
\email{szeptyck@yorku.ca}

\begin{abstract}
This paper addresses several questions of Feng, Gruenhage, and Shen which arose from Michael's theory of continuous selections from countable spaces. We construct an example of a space which is $L$-selective but not $\mathbb{Q}$-selective from $\mathfrak{d}=\omega_1$, and an $L$-selective space which is not selective for a $P$-point ultrafilter from the assumption of $\mathsf{CH}$. We also produce $\mathsf{ZFC}$ examples of Fr\'echet spaces where countable subsets are first countable which are not $L$-selective.
\end{abstract}

\maketitle

\section{Introduction}

Suppose $X,Y$ are topological spaces and $\varphi:Y\rightarrow \mathcal{P}(X)\setminus \{\emptyset\}$ is a map. A general question investigated in detail by Michael asks under what conditions it is possible to find a continuous $s:Y\rightarrow X$ so that $s(y)\in \varphi(y)$ for all $y\in Y$. It is natural to require some kind of continuity assumption for $\varphi$. So let $\mathcal{F}(X)$ be the set of all nonempty closed subsets of $X$ equipped with the Vietoris topology generated by the sets
\begin{enumerate}
\item $\{A\in \mathcal{F}(X):A\cap W\neq \emptyset\}$
\item $\{A\in \mathcal{F}(X):A\subseteq W\}$
\end{enumerate}
where $W$ ranges through open subsets of $X$. A map $\varphi:Y\rightarrow \mathcal{F}(X)$ is called \emph{lower semicontinuous} if it is continuous with respect to open sets of the first kind, i.e., if for every nonempty open $W\subseteq X$, the set $\{y\in Y: \varphi(y)\cap W\neq\emptyset\}$ is open in $Y$.

A space $X$ is $Y$-\emph{selective} if for every lower semicontinuous $\varphi:Y\rightarrow \mathcal{F}(X)$ there is a continuous selection $s:Y\rightarrow X$ for $\varphi$. $X$ is $L$-\emph{selective} if it is $(\omega+1)$-selective, where $\omega+1$ is a convergent sequence. $X$ is $C$-\emph{selective} if it is $Y$-selective for any countable regular space $Y$.

Michael \cite{M1} proved that every first countable space is $C$-selective. In \cite{FGS}, Feng, Gruenhage and Shen improved this result:
\begin{fact}
\begin{itemize}
\item[]
\item Every GO-space is $C$-selective.
\item Every $W$-space is $C$-selective.
\end{itemize}
\end{fact}

For $L$-selective spaces, they showed:
\begin{fact}
Every $L$-selective space $X$ has the $\alpha_1$-property, i.e., for any point $x\in X$ and any countable family $\mathcal{S}$ of sequences converging to $x$, there is a single sequence converging to $x$ which contains all but finitely many elements of each member of $\mathcal{S}$. 
\end{fact}
And furthermore, it is not difficult to show that if $X$ is $L$-selective, and $A\subseteq X$ is countable with $x\in\overline{A}$, then there is $B\subseteq A$ converging to $x$ (a slight weakening of the Fr\'echet property).

The present work addresses several questions of \cite{FGS}. They asked whether it is consistent that every $L$-selective space is $C$-selective, and constructed an example of an $L$-selective but not $\mathbb{Q}$-selective space under the assumption $\mathfrak{p}=\mathfrak{c}$. In Section \ref{lncscale}, we construct an example from $\mathfrak{d}=\omega_1$. In Section \ref{lncgap} we construct an $L$-selective space which is not selective for a $P$-point ultrafilter from the assumption of $\mathsf{CH}$. This construction uses the notion of a tight gap in $[\omega]^\omega$.

The authors of \cite{FGS} also asked whether it is consistent that every Fr\'echet space in which countable subspaces are first countable (CFC) is $L$-selective. This question was motivated on two fronts. Firstly, every GO-space and $W$-space is CFC; and secondly, there is a model of Dow and Steprans \cite{DS} where every countable Fr\'echet $\alpha_1$-space is first countable and hence every $L$-selective space is CFC. They produced an example of a Fr\'echet CFC space which is not $L$-selective from $\mathfrak{p}=\omega_1$. In Section \ref{cfcfnl}, we modify their example to waive the cardinal invariant assumption and provide a negative answer. Moreover, from an Aronszajn tree, we produce an example of size and character $\aleph_1$.

This work was initiated following G. Gruenhage's presentation at the 53rd Spring Topology and Dynamics Conference at the University of Alabama at Birmingham. The first named author would like to thank the organizers of the conference and acknowledges a travel award through NSF grant 1900727 that made possible his participation. The second author acknowledges support from NSERC Discovery Grant 2019-06356.

\section{Spaces obtained from an ultrafilter}
To investigate the problem of constructing an $L$-selective not $C$-selective space, it is helpful to have on hand some candidates for countable regular spaces $Y$ so that the space constructed is not $Y$-selective. A class of particularly simple spaces are those with a single nonisolated point. The collection of neighborhoods of the nonisolated point form a filter on the countable set of isolated points, so we call this space a \emph{filter space}. Given a filter $F$ on $\omega$, let $Y_F$ be the space with underlying set $\omega+1$ so that the points in $\omega$ are isolated and the neighborhoods of the point $\omega$ (which we call $\infty$ to avoid confusion with the set of isolated points) are given by $F$. We will use the abbreviation $F$-\emph{selective} to denote the property of being $Y_F$-selective.

To satisfy $F$-selectivity, it suffices to consider only those lower semicontinuous functions $\varphi$ so that $\varphi(\infty)$ is a singleton. To see this, take any lower semicontinuous function $\varphi:Y_U\rightarrow \mathcal{F}(X)$ and let $x\in \varphi(\infty)$. The function $\tilde{\varphi}$ defined so that $\tilde{\varphi}(n)=\varphi(n)$ for $n<\omega$ and $\tilde{\varphi}(\infty)=\{x\}$ remains lower semicontinuous, and a selection for $\tilde{\varphi}$ is also a selection for $\varphi$.

If $F$ and $G$ are filters on $\omega$, then a map $f:\omega\rightarrow \omega$ extends to a continuous map $\hat{f}:Y_F\rightarrow Y_G$ so that $\hat{f}(\infty)=\infty$ precisely when $f$ is a \emph{Kat\v{e}tov reduction}, i.e., $f^{-1}[B]\in F$ for all $B\in G$, and we write $G\le_K F$. The extension $\hat{f}$ is a quotient map precisely when $f$ is a \emph{Rudin--Keisler reduction}, i.e., $f^{-1}[B]\in F$ if and only if $B\in G$, and we write $G\le_{RK} F$. The Rudin-Keisler and Kat\v{e}tov orders coincide on the collection of ultrafilters. 

If $U$ is an ultrafilter, then being $U$-selective can be interpreted in terms of the ultrapower. Let $M=\mathrm{Ult}(H(\theta),U)$ and $j:H(\theta)\rightarrow M$ be the ultrapower embedding, where $\theta$ is a regular cardinal sufficiently large so that $H(\theta)$ includes all involved topological spaces and their subsets. A lower semicontinuous function $\varphi:Y_U\rightarrow \mathcal{F}(X)$ represents a closed set $[\varphi]$ in $j(X)$ so that for every open $W$ intersecting $\varphi(\infty)$, $M\Vdash j(W)\cap [\varphi]\neq \emptyset$. A continuous selection is a point in $[\varphi]\cap \bigcap_W j(W)$, where $W$ ranges over all open sets in $V$ intersecting $\varphi(\infty)$.

The next proposition shows that lower semicontinuous functions from an ultrafilter space whose range consists of bounded finite sets admits a selection. It is motivated by the proof of Proposition 3.2 in \cite{FGS} that $L$-selective spaces are $\alpha_1$. In that proof, any space $X$ which is not $\alpha_1$ is shown to admit a lower semicontinuous function mapping $\omega+1$ (with the ordinal topology) to pairs of elements of $X$.

\begin{proposition}
For every space $X$ and ultrafilter $U$, if $\varphi:Y_U\rightarrow \mathcal{F}(X)$ is lower semicontinuous and $\sup_{n<\omega} |\varphi(n)|<\omega$, then $\varphi$ has a continuous selection.
\end{proposition}
\begin{proof}
For each $n$, enumerate $\varphi(n)=\{x_{n,0},x_{n,1},\ldots,x_{n,k_n}\}$. Denote $\sup_n |\varphi(n)|$ by $N$. Pick $x\in \varphi(\infty)$. For each open $W$ neighborhood of $x$, let $X_W$ be the set of $k\le N$ so that $\{n:x_{n,k}\in W\}\in U$. 

We claim that $\bigcap_W X_W\neq \emptyset$. Otherwise, for each $k\le N$, there is some $W_k$ so that $\{n:x_{n,k}\in W_k\}\not\in U$. Now $\bigcap_{k\le N} W_k$ is an open neighborhood of $x$, but $\{n:\varphi(n)\cap \bigcap_{k\le N} W_k\neq\emptyset\}\not\in U$, contradicting lower semicontinuity.
\end{proof}

Let $S_\omega$ be the space on $(\omega\times\omega)\cup\{\infty\}$, thought of as $\omega$-many spines $\{n\}\times \omega$, $n<\omega$, where neighborhoods of $\infty$ are those sets which are cofinite in each spine. In other words, a neighborhood basis for $\infty$ is given by cofinite sets and open sets $W_f$, $f\in {}^\omega\omega$, where $W_f=\{\infty\}\cup \{(m,n):n>f(m)\}$. It is a typical example of a space which is Fr\'echet but not $\alpha_1$, and hence it is not $L$-selective.

However, we will show that $S_\omega$ is $U$-selective for $P$-point ultrafilters $U$. Recall that $U$ is a \emph{$P$-point ultrafilter} if every function $\omega\rightarrow \omega$ is either constant or finite-to-one on a set in $U$. Alternatively, for any partition of $\omega$ into subsets $P_n$, $n<\omega$, where $P_n\not\in U$, there are finite sets $p_n\subseteq P_n$ for each $n$ so that $\bigcup_n p_n\in U$. This notion occurs in several different places here, and seems to be important for this study.

\begin{proposition}
For every $P$-point ultrafilter $U$, $S_\omega$ is $U$-selective.
\end{proposition}
\begin{proof}
Suppose $\varphi:Y_U\rightarrow \mathcal{F}(S_\omega)$ is a lower semicontinuous function. We may assume that $\varphi(\infty)=\{\infty\}$.

 Let $M=\mathrm{Ult}(H(\theta),U)$ and $j:H(\theta)\rightarrow M$ be the ultrapower embedding. In the ultrapower, $j(S_\omega)$ can be described as $j(\omega)$ many spines, each of which is a copy of $j(\omega)$, together with $j(\infty)$.

\begin{claim} 
The intersection of all $j(W)$, where $W\in V$ ranges over open neighborhoods of $\infty$, is the set $\{(n,\beta):n<\omega \textrm{ and }\beta \textrm{ infinite}\}$. 
\end{claim}
\begin{proofofclaim}
If $W\in V$ and $n<\omega$, then $W$ contains all points on the $n$th spine above some natural number $m$, so $j(W)$ does as well and in particular contains $\{(n,\beta):\beta \textrm{ infinite}\}$. Suppose $(\alpha, \beta)\in j(S_\omega)$ with $\alpha$ infinite. Now $\alpha,\beta$ are represented by functions $a,b:\omega\rightarrow\omega$, respectively, and since $U$ is a $P$-point, we may take $a$ to be finite-to-one. Let $f:\omega\rightarrow\omega$ be defined as $f(i)=\max\{b(m):i=a(m)\}+1$, so that $f(a(n))>b(n)$ for all $n$. Then $(\alpha,\beta)\not\in j(W_f)$. This completes the proof of the claim.
\end{proofofclaim}

If the set $[\varphi]$ contains a point $(n,\beta)$, where $n<\omega$ is finite and $\beta\in j(\omega)$ is infinite, then this gives a continuous selection. Otherwise, by overspill---the principle that states that the set $\omega$ is not a member of $M$---
 $[\varphi]$ is bounded on $\{n\}\times \omega$ for each $n<\omega$, since if it were unbounded then it must also contain an infinite point in the $n$th spine. Let $h:\omega\rightarrow \omega$ be the function giving a bound for each $n$.
%
%

Since $[\varphi]$ is closed, there is some $g:j(\omega)\rightarrow j(\omega)$ in $M$ so that $[\varphi]$ intersects the $\alpha$th spine only below $g(\alpha)$. If $\langle g_n:\omega\rightarrow\omega\rangle$ is a sequence of functions which represents $g$ in $M$, then let $g':\omega\rightarrow \omega$ be such that $h< g'$ and $g_n<^*g'$ for each $n<\omega$. Let $\alpha\in j(\omega)$ be infinite, so $\alpha$ is represented by a finite-to-one function $a$. The set $\{n:g_n(a(n))\ge g'(a(n))\}$ is finite, therefore $\{n:g_n(a(n))< g'(a(n))\}\in U$ and so $g(\alpha)<j(g')(\alpha)$.

But then $\{n:\varphi(n)\cap W_{g'}\neq \emptyset \}\not\in U$, contradicting lower semicontinuity of $\varphi$. 
\end{proof}

\section{$L$-selective, not C-selective from a scale}\label{lncscale}

Suppose $\mathfrak{d}=\omega_1$. This is exemplified by a sequence of functions $\langle f_\alpha:\alpha<\omega_1\rangle$ so that $f_\alpha<^* f_\beta$ if $\alpha<\beta$, and for all $g\in {}^\omega\omega$ there is $\alpha<\omega_1$ so that $g<^* f_\alpha$.

Let $\{P_i:i<\omega\}$ be a partition of $\omega$ so that $P_i$ is infinite for each $i$. Let $$Z_\alpha=\{n<\omega: \exists i (n\in P_i \textrm{ and } n>f_\alpha(i))\}.$$

We define a space $X$ whose underlying set is $(\omega\times\omega_1)\cup\{\infty\}$, where all points of $\omega\times\omega_1$ are isolated, and a local subbase at $\infty$ is generated by the following sets:
\begin{itemize}
\item $X\setminus  (P_i\times \omega_1)$,
\item $X\setminus (Z_\alpha\times\alpha)$.
\end{itemize}
Let $\pi_0:X\rightarrow \omega$ and $\pi_1:X\rightarrow \omega_1$ denote the projection onto the first and second coordinates, respectively.

\begin{theorem}[$\mathfrak{d}=\omega_1$]
\label{lsnotus}
Suppose $U$ is a filter on $\omega$ so that $P_i\in I$ for all $i<\omega$ and $Z_\alpha\in I^+$ for all $\alpha<\omega_1$, where $I$ is the dual ideal of $U$. Then $X$ is Fr\'echet and $L$-selective but not $U$-selective.
\end{theorem}
\begin{proof}
First, we show that $X$ is not $U$-selective. Let $\psi:Y_U\rightarrow \mathcal{F}(X)$ be given by $\psi(n)=\{n\}\times \omega_1$ and $\psi(\infty)=\infty$. Note that $\psi(n)$ is closed for all $n< \omega$, since every $n$ is contained in some $P_i$ and then $\psi(n) \cap (X\setminus  (P_i\times \omega_1))=\emptyset$.

The function $\psi$ is lower semicontinuous since for every neighborhood $W$ of $\infty$, $\{n:(\{n\}\times\omega_1)\cap W\neq \emptyset\}\in U$. Suppose $s:Y_U\rightarrow X$ is a selection for $\psi$. Then there is $\alpha<\omega$ so that $\mathrm{ran}(s)\subseteq \alpha$. But $s^{-1}(X\setminus (Z_\alpha\times\alpha))=\omega\setminus Z_\alpha\not\in U$, so $s$ is not continuous.

Now we show that $X$ is $L$-selective. Suppose $\varphi:\omega+1\rightarrow \mathcal{F}(X)$ is lower semicontinuous. We will define a continuous selection $s:\omega+1\rightarrow X$. This is only nontrivial when $\infty\in \varphi(\infty)$ and in this case we will take $s(\infty)=\infty$. Furthermore, if $\infty\in \varphi(n)$ then we will take $s(n)=\infty$, so we assume that $\infty\not\in\varphi(n)$ for all $n<\omega$. Let $M\prec (H(\theta);\in,<_\theta)$ be a countable elementary submodel with $\varphi, \{f_\alpha:\alpha<\omega_1\} \in M$ and let $\delta=\sup(M\cap\omega_1)$. Define $$\pi_u(n)=\{i<\omega: \varphi(n)\cap \psi(i) \textrm{ is uncountable}\},$$ and $$\alpha^*=\sup\{\alpha: (\exists n)(\exists i\not\in \pi_u(n))\, (i,\alpha)\in \varphi(n)\}$$ and note that $\alpha^*\in M$, so $\alpha^*<\delta$.

\begin{claim}\label{cfc}
Every countable subspace of $X$ is first countable.
\end{claim}
\begin{proofofclaim}
Suppose $A\subseteq X$ is countable. Then let $\gamma=\sup \pi_1[A]$. We claim that the sets $A\setminus (P_i\times \omega_1)$ and $A\setminus (Z_\beta\times\beta)$, where $\beta\le \gamma$, form a local subbase at $\infty$ in $A$. For any $\alpha$, we show that $A\setminus (Z_\alpha\times\alpha)$ contains a finite intersection of the basic open sets. This is clear for $\alpha\le \gamma$, so assume that $\alpha>\gamma$. Since $Z_{\gamma}\subseteq^* Z_\alpha$, there is some $j$ so that $Z_{\gamma}\setminus Z_\alpha \subseteq \bigcup_{i\le j} P_i$. Then 
$$A\setminus (Z_\alpha\times\alpha)=A\setminus (Z_\alpha\times\gamma) \supseteq \left(A\setminus (Z_{\gamma}\times \gamma)\right)\cap \left(A\setminus (\bigcup_{i\le j} P_i \times \omega_1)\right).$$ This completes the proof of the claim.
\end{proofofclaim}


Continuing with the proof of $L$-selectivity, by using Claim \ref{cfc} in $M$ fix $$\emptyset=F_0\subseteq F_1\subseteq F_2 \subseteq\cdots$$ closed subsets of $X\setminus(\omega\times [\alpha^*,\omega_1))$ whose complements form a neighborhood basis of $\infty$ in $X\setminus(\omega\times [\alpha^*,\omega_1))$. Let $$V_j=(\bigcup_{i< j} P_i\times \omega_1)\cup F_j,$$ a closed subset of $X$. For each $j$, let $B_j=\{n:\varphi(n)\subseteq V_j \}$. By the lower semicontinuity of $\varphi$, $B_j$ is finite for each $j$. 
Let $b(n)$ be the least $j$ so that $n\in B_{j+1}$, if it exists, and undefined otherwise. Note that $b$ is finite-to-one.

Each of the following functions is in $M$, and therefore $<^* f_\delta$:
\begin{itemize}
\item for each $n$, the function $d_n:i \mapsto \min(\pi_u(n)\cap P_i)$ (and we interpret the minimum as $0$ if this set is empty).
\item the function $c:i\mapsto \max\{\min(\pi_u(n)\cap P_i): b(n)=i-1 \}$.
\end{itemize}

Let $i^*$ be large enough so that $c(i)<f_\delta(i)$ for $i> i^*$, and $i_n$ large enough so that  $d_n(i)<f_\delta(i)$ for $i> i_n$. Let $\langle \delta_n:n<\omega\rangle$ be increasing and cofinal in $\delta$. We will define a continuous selection $s$ for $\varphi$. There are three cases:

\emph{Case 1:} $b(n)$ exists and $P_{b(n)} \cap \pi_u(n)\neq \emptyset$. Choose $s(n)\in\varphi(n)\cap M$ with first coordinate equal to $\min(\pi_u(n)\cap P_{b(n)})$ and second coordinate above $\delta_n$ (and hence in the interval $(\delta_n,\delta)$), which exists by elementarity. 

\emph{Case 2:} $b(n)$ exists and $P_{b(n)} \cap \pi_u(n)= \emptyset$. Choose $s(n)$ to be a member of $\varphi(n)\setminus F_{b(n)}$ with second coordinate less than $\alpha^*$. This exists since otherwise $\varphi(n)\subseteq F_{b(n)}$, so the choice of $b(n)$ would not have been minimal. 

\emph{Case 3:} $b(n)$ doesn't exist. In this case, $\phi(n)\not\subseteq V_j$ for any $j<\omega$, so by the choice of $\alpha^*$ and the fact that the $V_j$ restrict to a neighborhood basis of $\infty$ in $X\setminus(\omega\times [\alpha^*,\omega_1))$, we have that $\{i: \pi_u(n)\cap P_i\neq\emptyset\}$ is infinite. Choose $s(n)\in \varphi(n)\cap M$ with first coordinate equal to $\min(\pi_u(n)\cap P_i)$ for some $i> \max\{n,i_n\}$, and second coordinate above $\delta_n$.

We now verify that $s$ is a converging sequence. By the proof of the earlier claim, it suffices to check that each of $P_i\times \omega_1$ and $Z_\beta\times\beta$, $\beta\le \delta$, contain $s(n)$ for at most finitely many $n$. Since $b(n)$ is finite-to-one, by the construction each subbasic set contains $s(n)$ for only finitely many $n$ from Case 2, and sets of the form $P_i\times \omega_1$ only contain $s(n)$ for at most finitely many $n$ from any case.  Since in Cases 1 and 3 we chose $s(n)$ to have second coordinate greater than $\delta_n$, $Z_\beta\times\beta$ contains $s(n)$ for only finitely many $n$ if $\beta<\delta$. 

It remains to show the result for $Z_\delta\times\delta$. Suppose $\pi_0(s(n))\in Z_\delta$. This means that $\pi_0(s(n))>f_\delta(i)$, where $\pi_0(s(n))\in P_i$. 

If $n$ is in Case 1, $\pi_0(s(n))=\min(\pi_u(n)\cap P_{b(n)})$. If $b(n)\ge  i^*$, then $\min(\pi_u(n)\cap P_{b(n)})<f_\delta(b(n))$ and so $b(n)< i^*$. Since $b$ is finite-to-one, there are only finitely many such $n$.

If $n$ is in Case 3, $\pi_0(s(n))= \min(\pi_u(n)\cap P_i)$ for some $i> \max\{n,i_n\}$. By the choice of $i_n$, $\min(\pi_u(n)\cap P_i)=d_n(i)<f_\delta(i)$, a contradiction.

This finishes the proof that $s$ is continuous.

Finally, we check that $X$ is Fr\'echet. Suppose $\infty\in\overline{A}$. Let $\pi_u(A)=\{i:A\cap\psi(i) \textrm{ is uncountable}\}$ and $\alpha^*=\sup \pi_1[\{(i,\alpha)\in A: i\not\in\pi_u(A)\}]$. If $\infty\in \overline{A\cap (\omega\times\alpha^*)}$ then we are done by Claim \ref{cfc}. Otherwise, by removing a countable set from $A$, we may assume that $\pi_u(A)=\pi_0[A]$. Let $\delta<\omega_1$ be such that:
\begin{itemize}
\item $\delta$ is a limit point of $\pi_1[A\cap \psi(i)]$ for every $i$,
\item $f_\delta$ dominates the function $i\mapsto \min(\pi_0[A]\cap P_i)$ (where we interpret the minimum as $0$ if this set is empty).
\end{itemize} 
It is then straightforward to verify that any finite union of sets of the form $P_i\times \omega_1$ and $Z_\beta\times\beta$, where $\beta\le \gamma$, do not contain all members of $A$.
\end{proof}

In \cite{FGS} it was observed that every countable metrizable space embeds as a closed subspace of $\mathbb{Q}$, and therefore if a space is not selective for some countable metrizable space, then it is not $\mathbb{Q}$-selective.
\begin{corollary}[$\mathfrak{d}=\omega_1$]
$X$ is not $\mathbb{Q}$-selective.
\end{corollary}
\begin{proof}
Take $U$ to be the dual filter of the ideal generated by $\{P_i:i<\omega\}$. Then $Y_U$ is countable metrizable, and $U$ satisfies the hypotheses of Theorem \ref{lsnotus} and since $X$ is not $U$-selective, it is not $\mathbb{Q}$-selective.
\end{proof}

\begin{corollary}[$\mathfrak{d}=\omega_1$]
\label{unotpp}
$X$ is not $U$-selective for any ultrafilter $U$ on $\omega$ which is not a $P$-point.
\end{corollary}
\begin{proof}
There is some partition $\{Q_i:i<\omega\}$ of $\omega$ so that each $Q_i$ is infinite but not in $U$ so that whenever $q_i\subseteq Q_i$ is finite, then $\bigcup_{i<\omega} q_i$ is not in $U$. By mapping $Q_i$ bijectively to $P_i$, $U$ is isomorphic to an ultrafilter which satisfies the hypotheses of Theorem \ref{lsnotus}. 
\end{proof}

\begin{corollary}
It is consistent relative to $\mathsf{ZFC}$ that there is an $L$-selective space $X$ which is not $U$-selective for any ultrafilter $U$.
\end{corollary}
\begin{proof}
The model obtained by iterated Silver forcing gives an example. In that model, $\mathfrak{d}=\omega_1$, and there are no $P$-points by a recent result of Chodounsk\'y and Guzm\'an \cite{CG}.
\end{proof}

\section{$L$-selective, not $C$-selective from a tight gap}\label{lncgap}

Gaps in $[\omega]^\omega$ have been used to find examples of interesting convergence properties starting with Nyikos \cite{N}. In this section, we will use $\mathsf{CH}$ to construct $L$-selective, not $C$-selective spaces based on a certain kind of gap. A tight gap is a sequence $\langle A_\alpha, B_\alpha:\alpha<\omega_1\rangle$ so that:
\begin{itemize}
\item $A_\alpha,B_\alpha$ are infinite subsets of $\omega$,
\item For all $\alpha<\beta$, $A_\alpha\subseteq^* A_\beta$ and $B_\beta\subseteq^* B_\alpha$ and $A_\beta\subseteq^* B_\alpha$,
\item If $E\subseteq^* B_\alpha$ for all $\alpha$, then there is $\beta$ so that $E\subseteq^* A_\beta$.
\item If $E\supseteq^* A_\alpha$ for all $\alpha$, then there is $\beta$ so that $E\supseteq^* B_\beta$.
\end{itemize}

A sequence $\langle A_\alpha, B_\alpha:\alpha<\gamma\rangle$ (where $\gamma\le \omega_1$) which satisfies the first two conditions of the definition of a tight gap is called a \emph{pre-gap}. We say that a set $E\subseteq \omega$ \emph{interpolates} the pre-gap $\langle A_\alpha, B_\alpha:\alpha<\gamma\rangle$ if $A_\alpha\subseteq^* E\subseteq^* B_\alpha$ for every $\alpha<\gamma$. The final two conditions of the definition (whose conjunction is called ``tightness'') ensure that no set interpolates the tight gap.

Given a tight gap $\mathcal{G}=\langle A_\alpha, B_\alpha:\alpha<\omega_1\rangle$, we can construct a space $X_\mathcal{G}$ with underlying set $\omega+1$, where the points of $\omega$ are isolated and the $B_\alpha$ generate the neighborhoods of $\omega$. We denote the point $\omega$ by $\infty$ to avoid confusion with the set of isolated points.

\begin{fact}
\begin{enumerate}
\item[]
\item An infinite set $E\subseteq \omega$ converges to $\infty$ if and only if there is some $\alpha$ so that $E\subseteq^* A_\alpha$.
\item $X_\mathcal{G}$ is a Fr\'echet $\alpha_1$ space.
\end{enumerate}
\end{fact}
\begin{proof}
(1) If $E\subseteq^* A_\alpha$ for some $\alpha$ then $E\subseteq^* B_\beta$ for all $\beta<\omega_1$, so $E$ converges to $\infty$. Conversely, suppose $E$ converges to $\infty$, so $E\subseteq^* B_\alpha$ for every $\alpha<\omega_1$. Then by tightness, we have that there is some $\alpha$ so that $E\subseteq^* A_\beta$.

(2) Suppose $\infty\in\overline{E}$ for some $E\subseteq \omega$. There must be $\alpha<\omega_1$ so that $E\cap A_\alpha$ is infinite, otherwise $\omega\setminus E \supseteq^* A_\alpha$ for all $\alpha$, and by tightness there is a $\beta$ so that $E\cap B_\beta$ is finite, contradicting $\infty\in\overline{E}$. By (1), $E\cap A_\alpha$ converges to $\infty$, and we conclude that $X_\mathcal{G}$ is Fr\'echet. If there are sequences $E_n$, $n<\omega$, all converging to $\infty$, then using (1) we have that for each $n$ there is $\alpha(n)$ so that $E_n\subseteq A_{\alpha(n)}$. Letting $\alpha=\sup_n \alpha(n)$, we have that $A_\alpha$ is sequence converging to $\infty$ which almost contains every $E_n$.
\end{proof}

In Corollary \ref{unotpp}, or by modifying Example 5.3 of \cite{FGS}, spaces which are $L$-selective but not $U$-selective for an ultrafilter $U$ are constructed. In those spaces, it seems essential that $U$ is not a $P$-point. The tight gap allows us to construct an example for a $P$-point ultrafilter.

\begin{theorem}[$\mathsf{CH}$]
There is a $P$-point ultrafilter $U$ and a tight gap $\mathcal{G}$ so that $X_\mathcal{G}$ is $L$-selective but not $U$-selective.
\end{theorem}

\begin{proof}
A \emph{pseudo-intersection} of a sequence $\langle B_\alpha:\alpha<\kappa\rangle$ is a set $B\subset \omega$ so that $B\subseteq^* B_\alpha$ for all $\alpha<\kappa$. A \emph{tower} is a sequence $\langle B_\alpha:\alpha<\kappa\rangle$ of subsets of $\omega$ so that $B_\beta\subseteq^* B_\alpha$ for $\alpha<\beta$, and there is no infinite pseudo-intersection.

By induction we will construct a tight gap $\langle A_\alpha,B_\alpha:\alpha<\omega_1\rangle$ and an ultrafilter $U$ generated by a tower $\langle U_\alpha:\alpha<\omega_1\rangle$. Using $\mathsf{CH}$, enumerate all infinite subsets of $\omega$ as $\langle E_\alpha:\alpha<\omega_1\rangle$. Let $\{\psi(i):i<\omega\}$ be a partition of $\omega$ so that $\psi(i)$ is infinite for each $i$. 

We will construct so that $A_\alpha\subseteq B_\alpha$ for all $\alpha$. In the construction, we maintain the following:
\begin{enumerate}
\item\label{cond:decides} Either $U_\alpha\cap E_\alpha=\emptyset$ or $U_\alpha\subseteq E_\alpha$.
\item\label{cond:finite} For each $i$, $B_\alpha\cap \psi(i)$ is finite,
\item\label{cond:Blarge}  For all $k$, $\{i: |B_\alpha \cap \psi(i)|>k\}\supseteq^* U_\alpha$,
\item\label{cond:Asmall}  $\{i: A_\alpha\cap \psi(i)=\emptyset\}\supseteq U_\alpha$.
\end{enumerate}
Condition (\ref{cond:decides}) ensures that $\langle U_\alpha:\alpha<\omega_1\rangle$ generates an ultrafilter. (\ref{cond:finite}) ensures that each $\psi(i)$ is closed, and holds for all $\alpha$ if it holds for $\alpha=0$.  (\ref{cond:Blarge}) ensures that $\psi\cup\{(\infty,\infty)\}$ is lower semicontinuous as a function $Y_U\rightarrow \mathcal{F}(X_\mathcal{G})$. (\ref{cond:Asmall}) is needed to make sure that no selection from $\psi$ gives a converging sequence (but is not sufficient to show that $\psi\cup\{(\infty,\infty)\}$ does not admit a continuous selection). 

At limit stages $\alpha$, we will construct $A_{\alpha}\subseteq B_{\alpha}$, both of which interpolate the gap constructed so far. 

\begin{claim}
Suppose $\gamma$ is a countable limit ordinal and $\langle A_\alpha, B_\alpha :\alpha<\gamma\rangle$ is a pre-gap which satisfies the conditions (\ref{cond:finite})--(\ref{cond:Asmall}). Then there are a pseudo-intersection $U_\gamma$ and interpolations $A_\gamma\subseteq B_\gamma$ which satisfy the conditions as well.
\end{claim}
\begin{proofofclaim}
It suffices to prove the claim for $\gamma=\omega$, as for arbitrary countable limit $\gamma$ we can take $\langle \alpha_n:n<\omega\rangle$ an increasing sequence cofinal in $\gamma$, and an interpolation for $\langle A_{\alpha_n}, B_{\alpha_n}:n<\omega\rangle$ also interpolates the full pre-gap. By making finite modifications to each set in the pre-gap, we may also assume that $A_m\subseteq A_n\subseteq B_m\subseteq B_n$ for all $m<n<\omega$.

Define $U_\gamma$ be a pseudo-intersection of $\{U_\beta:\beta<\gamma\}$ satisfying (\ref{cond:decides}). The set $U_\gamma$ is almost contained in $\{i:|B_n\cap \psi(i)|>k\}$ and in $\{i:A_n\cap \psi(i)=\emptyset\}$ for each $n,k<\omega$.

For each $n<\omega$, define inductively $g(n)$ so that $g(n)$ is the least natural number greater than or equal to $g(m)$ for all $m<n$ so that 
$$U_\gamma\setminus g(n)\subseteq \{i:|B_{n}\cap \psi(i)|>n\}\cap \{i:A_{n}\cap \psi(i)=\emptyset\}$$
Now define 
$$A_\omega=\bigcup_{n<\omega} (A_n\setminus \Psi(g(n)))$$
and
$$B_\omega=\bigcup_{n<\omega} (B_n\cap \Psi(g(n+1))),$$
where $\Psi(k)=\bigcup_{i\le k} \psi(i)$ for all $k$.

For any $n<\omega$, $A_n\setminus A_\omega\subseteq A_n\cap\Psi(g(n))$, which is finite by condition (\ref{cond:finite}); so $A_\omega\supseteq^* A_n$. Furthermore, $B_\omega\setminus B_n\subseteq \Psi(g(n))\cap B_0$, so $B_\omega\subseteq^* B_n$. Since $A_m\subseteq B_n$, for all $m,n$, we have $A_\omega\subseteq B_\omega$.

For any $k<\omega$, $n> k$, and $i\in U_\gamma$ at least $g(k)$, we have $|B_n\cap \psi(i)|> k$, so
$$\{i: |B_\omega \cap \psi(i)|>k\}\supseteq U_\gamma\setminus g(k).$$
For any $n<\omega$ and $i> g(n)$ we have $|A_n\cap \psi(i)|=\emptyset$, so $$\{i: |A_\omega \cap \psi(i)|=\emptyset\}\supseteq U_\gamma.$$ This finishes the proof of the claim.
\end{proofofclaim}

Let $\langle s_\alpha:\alpha<\omega_1\rangle$ enumerate all selections of $\psi$ and $\langle \varphi_\alpha:\alpha<\omega_1\rangle$ enumerate all functions $\varphi:\omega\rightarrow \mathcal{P}(\omega)\setminus\emptyset$. At stage $\alpha+1$, we construct $U_{\alpha+1}$, $A_{\alpha+1}, B_{\alpha+1}$ so that: 
\begin{enumerate}
\addtocounter{enumi}{4}
\item\label{cond:tightness2} there is $\beta\le \alpha+1$ so that $A_\beta \cap E_\alpha$ is infinite or $B_\beta\cap E_\alpha$ is finite.
\item\label{cond:tightnessb2} there is $\beta\le \alpha+1$ so that $E_\alpha\setminus A_\beta$ is finite or $E_\alpha\setminus B_\beta$ is infinite.
\item\label{cond:selective2} there is $\beta\le \alpha+1$ so that either
	\begin{itemize}
		\item for all $k<\omega$, $\{i: (A_\beta\setminus k)\cap \varphi_\alpha(i) = \emptyset \}$ is finite, or
		\item there is $k<\omega$ so that $\{i:(B_\beta\setminus k)\cap \varphi_\alpha(i)=\emptyset\}$ is infinite.
	\end{itemize}
\item\label{cond:killselection} $\{i:s_\alpha(i)\in B_{\alpha+1}\}\cap U_{\alpha+1}=\emptyset$.
\end{enumerate}

Conditions (\ref{cond:tightness2}) and (\ref{cond:tightnessb2}) ensure that the gap we construct at the end is tight. Condition (\ref{cond:selective2}) ensures that it is $L$-selective. In the first case, let $$k_i=\max\{k: (A_\beta\setminus k)\cap \varphi_\alpha(i) \neq \emptyset\},$$ and $i\mapsto \min((A_\beta\setminus k) \cap \varphi_\alpha(i))$ gives a continuous selection of $\varphi_\alpha$ defined on a cofinite set. In the second case, $B_\beta$ witnesses that the extension of $\varphi_\alpha$ to $\omega+1$ is not lower semicontinuous. Condition (\ref{cond:killselection}) ensures that $s_\alpha$ is not a continuous selection for $\psi$.

Now suppose we are at stage $\alpha+1$. Let us take care of conditions (\ref{cond:tightness2})--(\ref{cond:killselection}). There are two cases corresponding to the different options in (\ref{cond:selective2}).


%

\emph{Case 1}: There is an infinite set $Y\subseteq \omega$ so that
$$\sup\{|(B_\alpha\setminus\bigcup_{i\in Y} \varphi_\alpha(i))\cap \psi(n)|:n\in U_\alpha\}=\omega.$$

In this case, let $B'=B_\alpha\setminus\bigcup_{i\in Y} \varphi_\alpha(i)$. Let $\langle n_k:k<\omega\rangle$ be a sequence of distinct natural numbers such that $n_k\in U_\alpha$ and $|B'\cap \psi(n_k)|>k$, and let $U'=\{n_k:k<\omega\}.$

Let $Z_0=\emptyset$ if $B_\alpha\cap E_\alpha$ is finite, and otherwise let $Z_0$ be an infinite subset of $B_\alpha\cap E_\alpha$ so that $U'\setminus\psi^{-1}[Z_0]$ is infinite. Let $U''=U'\setminus\psi^{-1}[Z_0]$. Let $Z_1=\emptyset$ if $E_\alpha\setminus A_\alpha$ is finite, and otherwise let $Z_1$ be an infinite subset of $E_\alpha\setminus A_\alpha$ so that $U''\setminus\psi^{-1}[Z_1]$ is infinite. 

Take 
$$A_{\alpha+1}=A_\alpha\cup Z_0,$$
$$B_{\alpha+1}=A_\alpha\cup \left(B' \setminus (s_\alpha[\omega]\cup Z_1)\right)$$
and
$$U_{\alpha+1}=U''\setminus \psi^{-1}[Z_1].$$

By the choice of $Z_0$, (\ref{cond:Asmall}) holds. We check that (\ref{cond:Blarge}) holds. By the choice of $U_{\alpha+1}$, $U_{\alpha+1}\subseteq \{n:Z_1\cap \psi(n)=\emptyset\}$. Since $s_\alpha$ is a selection for $\psi$, for any $k$ we have
\begin{align*}\{n: |B_{\alpha+1} \cap \psi(n)|>k\}&\supseteq \{n\in U_{\alpha+1}: |B_{\alpha} \cap \psi(n)|>k+1\}\\
&=\{n_\ell\in U_{\alpha+1}:\ell>k+1\}\\
&\supseteq^* U_{\alpha+1}.
\end{align*}

Now (\ref{cond:tightness2}), (\ref{cond:tightnessb2}), (\ref{cond:selective2}), hold by the choices of $Z_0$, $Z_1$, and $B'$, respectively.  Finally, (\ref{cond:killselection}) holds since $s_\alpha[\omega]$ was subtracted off in the definition of $B_{\alpha+1}$ and so $\{n:s_\alpha(n)\in B_{\alpha+1}\}\cap U_{\alpha+1}\subseteq \{n\in U_\alpha:A_\alpha\cap\psi(n)\neq \emptyset\}=\emptyset$ by (\ref{cond:Asmall}) at $\alpha$.

\emph{Case 2}: For every infinite set $Y\subseteq \omega$, 
$$\sup\{|(B_\alpha\setminus\bigcup_{i\in Y} \varphi_\alpha(i))\cap \psi(n)|:n\in U_\alpha\}<\omega.$$

Let $U'\subset U_\alpha$ be such that $|U'|=|U_\alpha\setminus U'|=\aleph_0$. Then for any $k<\omega$, the set $$\{i:\varphi_\alpha(i)\cap(B_\alpha\setminus k)\subseteq \bigcup_{n\in U_\alpha\setminus U'} \psi(n)\}$$ is finite, otherwise there is some $k$ so that $Y=\{i:\varphi_\alpha(i)\cap(B_\alpha\setminus k)\subseteq \bigcup_{n\in U_\alpha\setminus U'} \psi(n)\}$ is infinite. But then $$\bigcup_{i\in Y} \varphi_\alpha(i)\cap B_\alpha\subseteq k\cup  \bigcup_{n\in U_\alpha\setminus U'} \psi(n),$$ so $B_\alpha\setminus\bigcup_{i\in Y} \varphi_\alpha(i)$ contains $\psi(n)\cap(B_\alpha \setminus k)$ for all $n\in U'$. Therefore $$\sup\{|(B_\alpha\setminus\bigcup_{i\in Y} \varphi_\alpha(i))\cap \psi(n)|:n\in U_\alpha\}=\omega,$$ so we would have been in Case 1.

Let $A'=A_\alpha\cup (B_\alpha\cap \bigcup_{n\in U_\alpha\setminus U'} \psi(n))$. Define $Z_0$, $Z_1$, and $U''$ in exactly the same way as the previous case, except using the new definition of $U'$. 

Take 
$$A_{\alpha+1}=A'\cup Z_0,$$
$$B_{\alpha+1}=A_\alpha\cup \left(B_\alpha \setminus (s_\alpha[\omega]\cup Z_1)\right)$$
and
$$U_{\alpha+1}=U''\setminus \psi^{-1}[Z_1].$$

Similarly as in the previous case, all of our conditions are satisfied for these choices.
\end{proof}
%

\section{$\mathrm{CFC}$, Fr\'echet, and not $L$-selective}\label{cfcfnl}

In \cite{FGS}, the question of whether $\mathrm{CFC}$ Fr\'echet spaces are $L$-selective is posed. There, they use $\mathfrak{p}=\omega_1$ to construct an example of a $\mathrm{CFC}$ Fr\'echet space which is not $L$-selective, but the question of whether there is a $\mathsf{ZFC}$ example is left open. However, a slight modification of their example answers the question.

Let $\langle A_\alpha:\alpha<\mathfrak{p}\rangle$ be a tower. Define a space $X$ on $(\omega\times\mathfrak{p})\cup\{\infty\}$, where all points of $\omega\times\mathfrak{p}$ are isolated, and a local subbase at $\infty$ is generated by the following sets:
\begin{itemize}
\item $X\setminus  (\{n\}\times \mathfrak{p})$,
\item $X\setminus (A_\alpha\times\alpha)$,
\item $X\setminus (B\times \alpha)$, where $\cf(\alpha)>\omega$ and $B$ is a pseudo-intersection of $\{A_\beta:\beta<\alpha\}$. 
\end{itemize}

\begin{theorem}
$X$ is a $\mathrm{CFC}$, Fr\'echet, and not $L$-selective space.
\end{theorem}
\begin{proof}
Let $\varphi:\omega+1\rightarrow X$ be defined by $\varphi(\omega)=\infty$ and $\varphi(n)=\{n\}\times \mathfrak{p}$. Then $\varphi$ is lower semicontinuous, since the neighborhood filter at $\infty$ is contained in the dual filter of the ideal $\mathrm{Fin}\times\mathrm{Bounded}$, where $\mathrm{Fin}$ is the ideal of finite subsets of $\omega$ and $\mathrm{Bounded}$ is the ideal of bounded subsets of $\mathfrak{p}$. Suppose $s:\omega+1\rightarrow X$ is a selection for $\varphi$. Then there is $\alpha<\omega$ so that $\mathrm{ran}(s\rest\omega)\subseteq \omega\times\alpha$. But $s^{-1}(X\setminus (A_\alpha\times\alpha))=\omega\setminus A_\alpha$ is co-infinite, so $s$ is not continuous. Therefore, $X$ is not $L$-selective.

\begin{claim}
Every countable subspace of $X$ is first countable.
\end{claim}
\begin{proofofclaim}
Suppose $A\subseteq X$ is countable, and to avoid trivialities we may assume that $\infty\in A$. Let $a$ be the set of successors of elements of $\pi_1[A]$. We check that the sets $A\setminus (\{n\}\times \mathfrak{p})$ and $A\setminus (A_\alpha \times\alpha)$, where $\alpha\in \overline{a}$ (closure in the order topology on $\kappa$), form a neighborhood basis for $\infty$ in $A$. 

To see this, we check that the intersection of $A$ with any of the subbasic open neighborhoods of $\infty$ contains a finite intersection of sets of the form $A\setminus (\{n\}\times \mathfrak{p})$ and $A\setminus (A_\alpha \times\alpha)$, $\alpha\in \overline{a}$. The relevant cases are the neighborhoods $V$ of the form $X\setminus (A_\beta\times\beta)$ and $X\setminus (B\times \beta)$, $\beta\ge \min(a)$, since if $\beta<\min(a)$ then $A\cap V=A$. 

For a neighborhood $X\setminus (A_\beta\times\beta)$, let $\gamma\in\overline{a}$ be maximum so that $\gamma\le \beta$. Then $A_\beta\subseteq^* A_\gamma$ and $A\cap (\omega\times\beta)=A\cap(\omega\times\gamma)$, so $A\setminus (A_\beta\times\beta)$ contains the intersection of $A\setminus (A_\gamma\times \gamma)$ with finitely many sets of the form  $A\setminus (\{n\}\times \mathfrak{p})$. 

Now consider a neighborhood $X\setminus (B\times \beta)$, where  $\cf(\beta)>\omega$ and $B$ is a pseudo-intersection of $\{A_\gamma:\gamma<\beta\}$. Let $\gamma\in\overline{a}$ be maximum so that $\gamma\le \beta$. Since $A$ is countable, $\overline{a}$ consists of ordinals which are either successors or have countable cofinality, so $\gamma<\beta$. By assumption, $B\subseteq^* A_\gamma$, so $A\setminus (B\times \beta)$ contains the intersection of $A\setminus (A_\gamma\times \gamma)$ with finitely many sets of the form $A\setminus (\{n\}\times \mathfrak{p})$. 
\end{proofofclaim}

\begin{claim}
$X$ is Fr\'echet.
\end{claim}
\begin{proofofclaim}
Suppose $\infty\in\overline{A}$. By the previous claim, it suffices to find a countable subset of $A$ which accumulates to $\infty$. We may assume that $\pi_1[A]$ has no maximum element and let $\alpha=\sup\pi_1[A]$. We may further assume that for every $\beta<\alpha$, $\infty\not\in \overline{A\cap(\omega\times \beta)}$.

If $\cf(\alpha)=\omega$, let $\langle \alpha_n:n<\omega\rangle$ be a sequence cofinal in $\alpha$. For each $n$, there must be some $x_n\in A$ such that $\pi_0(x_n)\not\in n\cup A_\alpha$ and $\pi_1(x_n)\ge \alpha_n$, since $\infty\not\in \overline{A\cap(\omega\times \alpha_n)}$. It can be easily checked that $\{x_n:n<\omega\}$ has finite intersection with any set of the form $\{n\}\times \mathfrak{p}$ or $B\times\beta$ for $\beta<\alpha$. And if $\beta\ge\alpha$, the closed sets of the form $B\times \alpha$ have $B\subseteq^* A_\alpha$, so again $B\times \alpha$ and $\{x_n:n<\omega\}$ have finite intersection, and we conclude that $\{x_n:n<\omega\}$ converges to $\infty$.

If $\cf(\alpha)>\omega$, then for every $\beta<\alpha$, $A\cap(\omega\times \beta)$ is covered by the union of $B\times \beta$ (where either $B=^* A_\alpha$ or $B$ is a pseudo-intersection of $\{A_\delta:\delta<\beta\}$ so that $A_\beta\subseteq B$) together with finitely many basic closed sets of the form $B_\gamma\times \gamma$, $\gamma<\beta$. Let $\gamma(\beta)$ be the maximum of the finitely many $\gamma$ which appear in this union. By Fodor's lemma, there is an unbounded set $S\subseteq \alpha$ so that $\gamma\rest S$ is constant, say with value $\gamma^*$. Since $\infty\not\in \overline{A\cap (\omega\times\gamma^*)}$, we may remove $\omega\times\gamma^*$ from $A$ and assume that for every $\beta<\alpha$, there are $\delta\in S$ and $B$ a pseudo-intersection of $\{A_\xi:\xi<\delta\}$ so that $A\cap (\omega\times\beta)\subseteq B\times\delta$.

By removing $\omega\times\alpha^*$, where $\alpha^*$ is the supremum of the ordinals $\sup\{\xi:(n,\xi)\in A\}$ which are less than $\alpha$, we may assume that for every $n\in \pi_0[A]$, $\{\xi:(n,\xi)\in A\}$ is unbounded in $\alpha$. By the work of the previous paragraph, $\pi_0[A]$ is contained in a pseudo-intersection $B^*$ of $\{A_\xi:\xi<\alpha\}$. Therefore, $A$ is contained in the closed set $B^*\times \alpha$,  a contradiction.
\end{proofofclaim}

\end{proof}

The space constructed above is of cardinality and character $\mathfrak{p}$. We can also construct an example of size $\aleph_1$, even if $\mathfrak{p}>\omega_1$.

\begin{theorem}
There is a Fr\'echet $\mathrm{CFC}$ space which is not $L$-selective with cardinality $\aleph_1$.
\end{theorem}
\begin{proof}
Let $T$ be an Aronszajn tree. Define an order-preserving map $\sigma$ from $(T,<)$ into $([\omega]^{\omega},\supseteq^*)$ by induction on the level, so that on each level of $T$ the image of $\sigma$ consists of pairwise almost-disjoint sets.

More explicitly, let $\sigma(\emptyset)=\omega$ and assume for $\alpha<\omega_1$ that $\sigma$ is defined on all levels $<\alpha$. If $\alpha=\beta+1$ is successor, then for each $s\in T_\beta$, partition $\sigma(s)$ into pairwise disjoint infinite sets $\{ a_{s,i}:i<\omega\}$ and enumerate the set of $t>s$ on level $\alpha$ as $\{t_i:i<i^*\le\omega\}$. Then put $\sigma(t_i)=a_{s,i}$. If $\alpha$ is limit, then for each $t\in T_\alpha$, let $\sigma(t)$ be a pseudo-intersection of $\{\sigma(s):s<t\}$, so that $\sigma(t)\subseteq^* \sigma(s)$ for all $s<t$. 

Let $I_\alpha$ be the ideal on $\omega$ generated by $\{\sigma(t):t\in T_\alpha\}$, i.e., $a\in I_\alpha$ if and only if $a$ is contained in a finite union of sets of the form $\sigma(t), t\in T_\alpha$.

Notice that the sequence $\{I_\alpha:\alpha<\omega_1\}$ is decreasing. In addition, we have:
\begin{claim}
For every infinite set $B$, there is some $\alpha<\omega_1$ so that $B\not\in I_\alpha$.
\end{claim}
\begin{proofofclaim}
Assume towards a contradiction that there is a set $B$ so that $B\in I_\alpha$ for every $\alpha<\omega_1$. Let $T_B=\{t\in T: \sigma(t)\cap B \textrm{ is infinite}\}$. $T_B$ is a subtree of $T$, since if $t\in T_B$ and $s<t$, then $\sigma(s)\supseteq^*\sigma(t)$ and $\sigma(s)\cap B$ is infinite, so $s\in T_B$. The levels of $T_B$ are finite, since $B\in I_\alpha$ and hence $B$ is covered by a finite union of sets from the pairwise almost-disjoint family $\{\sigma(t):t\in T_\alpha\}$. $T_B$ has height $\omega_1$.

Since $T_B$ is a subset of the Aronszajn tree $T$, $T_B$ has no cofinal branch. But any tree with finite levels and height $\omega_1$ must have a cofinal branch, a contradiction (this fact can be seen by taking a uniform ultrafilter $U$ on $T_B$, so $\{s\in T_B: \{t\in T_B: s<t \}\in U\}$ determines a cofinal branch).
\end{proofofclaim}

Now we can define a space $X_T$ whose underlying set is $\{\infty\}\cup(\omega\times\omega_1)$ so that points of $\omega\times\omega_1$ are isolated and the neighborhood filter for $\infty$ is generated by sets of the form:
\begin{enumerate}
\item $X_T\setminus (n\times \omega_1)$, where $n<\omega$,
\item $X_T\setminus (B\times \alpha)$, where $\alpha<\omega_1$ and $B\in I_\alpha$.
\end{enumerate}

$X_T$ is not $L$-selective, as witnessed by the lower semicontinuous function $\varphi:\omega+1\rightarrow \mathcal{F}(X_T)$ where $\varphi(\omega)=\infty$ and $\varphi(n)=\{n\}\times \omega_1$. Any selection has countable range bounded by some $\alpha<\omega_1$, and $I_\alpha$ contains an infinite set, so there are no continuous selections.

It is clear that $X_T$ is $\mathrm{CFC}$. It remains to check that it is Fr\'echet, and since $X_T$ is $\mathrm{CFC}$, it suffices to show that $X_T$ has countable tightness.

Suppose $\infty\in\overline{A}$. We will show that $\infty\in\overline{A'}$ for some countable $A'\subseteq A$. Let $\pi_0:X_T\rightarrow \omega$ and $\pi_1:X_T\rightarrow \omega_1$ denote the first and second projection, respectively. We may assume that the set $A_u=\{n:A\cap (\{n\}\times\omega_1) \textrm{ is uncountable}\}$ is equal to $\pi_0[A]$, since either $$\infty \in \overline{A\cap (A_u\times\omega_1)}$$ or $$\infty \in \overline{A\cap ((\omega\setminus A_u)\times\omega_1)},$$ and in the second case we are already done.

By the earlier claim, there is some $\alpha<\omega_1$ so that $\pi_0[A] \not\in I_\alpha$. By increasing $\alpha$, we can take $\alpha$ to be a limit point of $\pi_1[A\cap (\{n\}\times \omega_1)]$ for all $n\in \pi_0[A]$. Then $\infty \in \overline{A\cap (\omega\times\alpha)}$, since any basic open neighborhood of $\infty$ either contains $\omega\times\beta$ for some $\beta<\alpha$ or $B\times\alpha$ for some $B\subseteq \pi_0[A]$ which is positive for $I_\alpha$.
\end{proof}

\bibliography{selective}{}
\bibliographystyle{plain}

\end{document}